\documentclass{amsart}
\usepackage{amsmath}
\usepackage{amssymb}
\usepackage{amsfonts}
\usepackage{amsthm}
\usepackage{mathrsfs}
\usepackage[all]{xy}
\usepackage[pdftex]{graphicx}
\usepackage{color}
\usepackage{cite}
\usepackage{url}

\newtheorem{theorem}{Theorem}[section]
\newtheorem{lemma}[theorem]{Lemma}

\theoremstyle{definition}
\newtheorem{definition}[theorem]{Definition}
\newtheorem{example}[theorem]{Example}

\theoremstyle{remark}
\newtheorem{remark}[theorem]{Remark}

\numberwithin{equation}{section}

\DeclareMathOperator{\Aut}{Aut}

\DeclareMathOperator{\Gal}{Gal}

\newcommand{\eps}{\varepsilon}

\newcommand{\CC}{\mathbb{C}}

\newcommand{\PP}{\mathbb{P}}
\newcommand{\QQ}{\mathbb{Q}}
\newcommand{\RR}{\mathbb{R}}
\newcommand{\ZZ}{\mathbb{Z}}

\newcommand{\mcM}{\mathcal{M}}

\newcommand{\Qbar}{\overline{\QQ}}

\begin{document}

\title{Period computations for covers of elliptic curves}

\author{Simon Rubinstein-Salzedo}
\address{Department of Mathematics, Dartmouth College, Hanover, NH 03755}
\email{simon.rubinstein-salzedo@dartmouth.edu}

%    General info

\date{\today}


\begin{abstract}
In this article, we construct algebraic equations for a curve $C$ and a map $f$ to an elliptic curve $E$, with pre-specified branching data. We do this by determining certain relations that the periods of $C$ and $E$ must satisfy and using these relations to approximate the coefficients to high precision. We then conjecture which algebraic numbers the coefficients are, and then we prove this conjecture to be correct.
\end{abstract}

\maketitle

\tolerance 5000

\section{Introduction}

Transcendental methods have long been used to solve problems in algebraic number theory. The most celebrated instances of such methods have been in Diophantine approximations and closely related problems, including Baker's solution of the class number one problem for imaginary quadratic fields in \cite{Baker67}.

But transcendental techniques have been used in other ways as well. In particular, the exact values of several integrals and infinite sums have been conjectured by first computing them to high precision and then recognizing the numbers that appear. This article is in the spirit of that approach: we will compute certain numbers, which we know \emph{a priori} to be algebraic, to high precision and then recognize them as roots of certain polynomials.

In this article, we will be interested in constructing a curve $C$ of genus 2 that admits a map $f$ of degree 5, branched at one point with local monodromy of type a 3-cycle, to a particular elliptic curve $E$, defined over $\QQ$. We shall do this by utilizing transcendental methods. By standard arguments presented in \S\ref{existence}, we can show that $C$ is defined over some number field $K$ and bound the degree of $K$. Then, we can determine various relations that the periods of a certain differential on $C$ must satisfy and use these relations to approximate the coefficients of $C$ and $f$ to very high precision. Once we have determined the coefficients to sufficiently high precision, we can use the LLL algorithm (see \cite{LLL82}) or continued fractions (see \cite{BvdP95}) in order to guess which numbers they are. It is then easy to prove that our answer is correct.

It is worth mentioning briefly some other places in mathematics where branched covers of curves come up, even though we do not need input from them in this article. To number theorists, branched covers are important in the study of (for example) \'etale cohomology and the \'etale fundamental group, which in turn provide information about points defined over arithmetically interesting fields, of our base curve. They can also be used to construct number fields with limited ramification, as is shown by Beckmann in \cite{Beck89} and used in practice by Roberts in \cite{Rob04} in the case of covers of $\PP^1-\{0,1,\infty\}$.

There is also a topological or complex analytic side to the story. A branched cover of an elliptic curve branched at one point is known as an origami. Origamis can be represented pictorially as in Figure \ref{swisscross}, and the topological viewpoint has been much studied, for example by Herrlich and Schmith\"usen in \cite{HS09}. Conversely, if $E$ is the elliptic curve $\CC/\ZZ[i]$, a square tile arrangement like the one shown in Figure \ref{swisscross} can be interpreted as a smooth curve $C$ (say of genus $g$) with a differential $\omega$ and a map $f:C\to E$ satisfying $f^\ast\frac{dx}{y}=\omega$; hence the square tiling determines a flat structure on a curve of genus $g\ge 1$, with finitely many cone points. This connection is surveyed, for instance, by Zorich in \cite{Zor06}. As we modify the flat structure on $E$ and carry the flat structure on $C$ along with it, we obtain a family of curves $C_t$ inside the moduli space $\mcM_g$. Such a curve, if written down explicitly with algebraic equations, is also of arithmetical interest, as Roberts's techniques are applicable here as well.

Another facet of the subject comes from the study of billiards on polygons, which has been much studied especially by McMullen in the genus two case, for example in \cite{McM05}. In this part of the story, the objects of study are the geodesics on a curve $C$; for example, we might wish for asymptotics on how many (foliations of) closed geodesics there are of length $\le x$, or to know whether a trajectory on a flat surface must be either a closed geodesic or equidistributed over $C$.

The structure of the paper is as follows. In \S\ref{thmstatement}, we present our main result: we produce an explicit example of a branched cover of an elliptic curve ramified at one point with local monodromy of type a 3-cycle, and prove that it actually is what we claim. Starting in \S\ref{existence}, we forget that we have found this cover and begin the path to its construction from scratch. In \S\ref{existence}, we explain why such a cover must exist and show that its coefficients can be assumed to be algebraic. In \S\ref{coeffs}, we use period computations to approximate the coefficients to very high precision. In \S\ref{algebraize}, we take a step back to discuss methods for detecting algebraic numbers given many digits of their decimal expansions, and we apply this to the coefficients we found in \S\ref{coeffs}. Then, in \S\ref{eqnmap}, we explain how to find the equation of the map $f:C\to E$, again by approximating the coefficients to high precision and then algebraizing. In \S\ref{further}, we suggest future directions this project can go.

\section{A branched cover of an elliptic curve} \label{thmstatement}

In this section, we state and prove the main theorem of this article:

\begin{theorem} \label{311curve} The genus-2 curve \[C:-\sqrt{5}y^2=x(x-1)(x-\kappa)(x-2\kappa+1)(x-2\kappa),\qquad \kappa=81+36\sqrt{5},\] admits a degree-5 map $f$ to the elliptic curve \[E:y^2=x^3-x\] branched only above $\infty$ with a triple point above $\infty$. The map is given by \[f(x,y)=(g(x),h(x)y),\] where \begin{align*} g(x) &= \frac{x^5+a_4x^4+a_3x^3+a_2x^2+a_1x+a_0}{b_2x^2+b_1x},\\ h(x) &= \frac{x^6+c_5x^5+c_4x^4+c_3x^3+c_2x^2+c_1x+c_0}{dx^2(x-2\kappa)^2},\end{align*} and \begin{align*} a_4 &= -45(9+4\sqrt{5}) \\ a_3 &= 660(161+72\sqrt{5}) \\ a_2 &= -3240(2889+1292\sqrt{5}) \\ a_1 &= 1980(51841+23184\sqrt{5}) \\ a_0 &= -324(930249+416020\sqrt{5}) \\ b_2 &= -100\sqrt{5}(2889+1292\sqrt{5}) \\ b_1 &= 1800\sqrt{5}(51841+23184\sqrt{5}) \\ c_5 &= -54(9+4\sqrt{5}) \\ c_4 &= 1030(161+72\sqrt{5}) \\ c_3 &= -7920(2889+1292\sqrt{5}) \\ c_2 &= 18780(51841+23184\sqrt{5}) \\ c_1 &= 216(930249+416020\sqrt{5}) \\ c_0 &= -1944(16692641+7465176\sqrt{5}) \\ d &= 1000\sqrt{5}(219602+98209\sqrt{5}). \end{align*} \end{theorem}

\begin{proof} In order to show that $f(x,y)$ gives a map $E\to C$, we need only check that whenever $(x,y)$ is a point on $C$, then $(g(x),h(x)y)$ is a point on $E$. This amounts to checking that if \[-\sqrt{5}y^2=x(x-1)(x-\kappa)(x-2\kappa+1)(x-2\kappa),\] then \[-\frac{1}{\sqrt{5}}h(x)^2y^2=g(x)^3-g(x),\] or \[h(x)^2x(x-1)(x-\kappa)(x-2\kappa+1)(x-2\kappa)=g(x)^3-g(x).\] This now amounts to verifying a formal identity of rational functions.

To show that the map is only branched above $\infty$, let $\omega=\frac{dx}{y}$ be the invariant differential on $E$. This differential has no zeros or poles. The pullback $f^\ast\omega\in\Omega^1_C$ is a holomorphic differential and therefore has zeros exactly at the branch points of $f$. Furthermore, if $f^\ast\omega$ has a zero of order $e-1$ at a point $P$, then $P$ has ramification index $e$. We compute $f^\ast\omega$ to be a constant multiple of $\frac{dx}{y}$, which has a double zero at $\infty$ and no other zeros or poles. Hence $f$ is branched only above $\infty$, with a triple point above $\infty$. \end{proof}

This proof, while completely correct, is nonetheless not entirely satisfactory, as it does not provide any mechanism by which one can discover the equations for $C$ and $f$ from scratch. In the remainder of the article, we present a method, based on transcendental techniques, to find these equations.

\section{The existence of $C$} \label{existence}

In this section, we explain why we know \emph{a priori} that such a curve $C$ and a map $f$ exist and can be defined over $\Qbar$, and we see how we can bound the degree of a number field over which they are defined.

We first show that $C$ and $f$ can be defined over $\CC$. Since there is an equivalence of categories between the category of compact Riemann surfaces with holomorphic maps and the category of smooth projective complex curves with algebraic maps, it suffices to show that the pair $(C,f)$ exists in the category of compact Riemann surfaces. As Riemann surfaces, $E\cong\CC/\ZZ[i]$. Hence, we can view $E$ as a fundamental parallelogram for the lattice $\ZZ[i]\subset\CC$, with opposite edges identified. We choose the square with vertices $0,1,1+i,i$ as a fundamental parallelogram for $\ZZ[i]$.

In order to construct $C$, we find a collection of five squares, together with gluing information, so that when we glue the appropriate edges, we obtain an orientable genus-2 surface with the correct ramification type. A diagram is given in Figure \ref{swisscross}. The map $f$, shown in Figure \ref{deg5map}, takes a point in some square of $C$ to the corresponding point in the square of $E$. It is branched only above the vertex of $E$, and it has the desired ramification type.

\begin{figure} \begin{center} \includegraphics[height=1in]{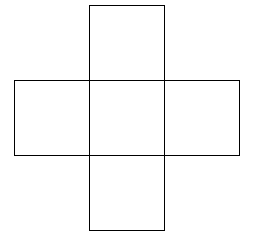} \end{center} \caption{This diagram gives the charts for a genus-2 curve with a degree-5 map to an elliptic curve. Here, we identify opposite edges. Above the branch point, we have one triple point, which corresponds to the four central vertices, which are glued together and two unramified points, which correspond to the other eight vertices. This diagram is sometimes called a Swiss cross.} \label{swisscross} \end{figure}

\begin{figure} \begin{center} \includegraphics[height=1in]{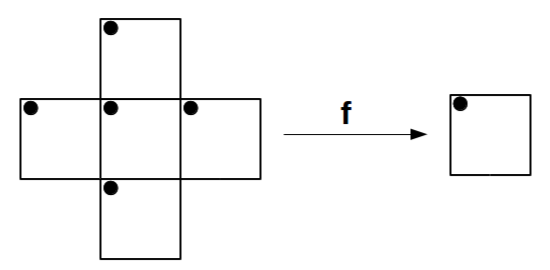} \end{center} \caption{This diagram shows all the preimages under $f$ in $C$ of the marked point in $E$.} \label{deg5map} \end{figure}

Now we know that $C$ and $f$ can be defined by equations over $\CC$. We will see next how to show that they can in fact be defined over $\Qbar$. To do this, we need the following definition.

\begin{definition} Let $(C,f)$ and $(D,g)$ be two branched covers of $E$. We say that they are isomorphic if there exists an isomorphism $h:C\to D$ that makes the diagram \[\xymatrix{C\ar[rr]^h\ar[dr]_f && D\ar[dl]^g \\ & E}\] commute. \end{definition}

The isomorphism classes of branched covers of $E$ can be neatly parametrized by group-theoretic data. The following result is a special case of the Riemann Existence Theorem; a complete statement can be found in \cite{Don11}, \S 4.2.2.

\begin{theorem} The isomorphism classes of connected branched covers of $E$ of degree $n$ branched only above $\infty$ are in bijection with equivalence classes of pairs $(\sigma,\tau)$ of permutations in $S_n$ which generate a transitive subgroup of $S_n$. Two pairs of permutations $(\sigma_1,\tau_1)$ and $(\sigma_2,\tau_2)$ are equivalent if there is some $\gamma\in S_n$ so that $\sigma_2=\gamma^{-1}\sigma_1\gamma$ and $\tau_2=\gamma^{-1}\tau_1\gamma$. \end{theorem}

The most important consequence of this result for us is that there are only finitely many isomorphism classes of branched covers of $E$ of degree $n$ branched only above $\infty$. In fact, we can even bound the number, and in low-degree cases compute it exactly.

In this article, we are concerned with degree-5 covers branched only above $\infty$, with local monodromy of type a 3-cycle. The equivalence classes of such covers are therefore in bijection with equivalence classes of pairs of permutations $(\sigma,\tau)$ in $S_5$, generating transitive permutation groups, and for which the commutator of $\sigma$ and $\tau$ is a 3-cycle. There are 27 such equivalence classes.

Since $E$ is defined over $\QQ$ and the point at $\infty$ is in $E(\QQ)$, the group $\Aut(\CC/\QQ)$ acts on the set of isomorphism classes of branched covers of $E$ of degree $n$ branched only above $\infty$, by acting on the coefficients of a representative $(C,f)$ of an isomorphism class of covers. Since there are only finitely many such isomorphism classes, the action of $\Aut(\CC/\QQ)$ on the set of isomorphism classes of covers factors through $\Gal(K/\QQ)$, for some number field $K$.

At this stage, we need a pair of definitions.

\begin{definition} \begin{enumerate} \item The field of moduli of an isomorphism class of a cover $(C,f)$ (or just of the cover) is the fixed field of the stabilizer in $\Aut(\CC/\QQ)$ of the isomorphism class of $(C,f)$. \item A field $K$ is a field of definition for an isomorphism class $[(C,f)]$ of a cover (or simply of a cover) if there is some representative $(C,f)$ of that isomorphism class so that the equations for $C$ and $f$ have coefficients in $K$. \end{enumerate} \end{definition}

Note that there are many fields of definition for one (isomorphism class of) cover. In the setting above, $K$ contains the field of moduli for $C$, for all such covers $C$.

The field of moduli is the intersection (inside a fixed algebraic closure $\Qbar$ of $\QQ$) of all fields of definition of $(C,f)$ algebraic over $\QQ$. However, the field of moduli $K$ for a cover $C$ is not always itself a field of definition. That is, there is not always a model of the pair $(C,f)$ with coefficients in $K$, and an example can be found in \cite{CG94}. However, it is frequently the case that the field of moduli is a field of definition, and there are several sufficient conditions for this to be the case. D\`ebes and Douai in \cite{DD97} show that the obstruction for the field of moduli to be a field of definition is measured by the nonvanishing of several cohomology classes, and they deduce the following as a corollary:

\begin{theorem}[D\`ebes-Douai] If $(C,f)$ is a cover of a curve $E$, and the monodromy group of $f$ has trivial center, then the field of moduli is a field of definition. \end{theorem}

Hence $C$ and $f$ can be defined over a number field, and even one whose degree we can bound effectively.

Our approach now will be to use period calculations to compute the coefficients of $C$ and $f$ to sufficiently high precision, so that we can detect which algebraic numbers they are. This will be done in the next section.

\section{Finding the coefficients} \label{coeffs}

Before we begin computing periods, we must place some normalization conditions on the curve $C$ and the map $f$ we are looking for. All genus-2 curves are hyperelliptic and hence can be written in the form $y^2=q(x)$, a quintic polynomial in $x$. By composing with a suitable linear transformations on the $x$- and $y$-coordinates if necessary, we may assume that $q$ is monic and has roots at 0 and 1. Hence we may write $C$ as \[y^2=x(x-1)(x-r_3)(x-r_4)(x-r_5)=:q(x).\] Let us also set $r_1=0$ and $r_2=1$ for convenience. Let us decide that $f$ will be branched above the point $O=(0:1:0)$ at $\infty$ of $E$. We can also narrow down the list of possible choices of critical points on $C$.

\begin{lemma} \label{Wptlemma} If $C$ is a genus-2 curve, $E$ is an elliptic curve, and $f:C\to E$ is a map with a unique critical point $p$ in $C$, then $p$ is a Weierstra\ss\ point of $C$. \end{lemma}

\begin{proof} Let $K_C$ and $K_E=0$ denote canonical divisors of $C$ and $E$, respectively. By the Riemann-Hurwitz formula, $K_C$ is linearly equivalent to $f^\ast K_E+2[p]=2[p]$. Hence $\dim H^0(C,2[p])=\dim H^0(C,K_C)=2$, so there is a meromorphic function on $C$ with a double pole at $p$ and no other poles, i.e., $p$ is a Weierstra\ss\ point. \end{proof}

%\begin{proof} Suppose $C$ is written in the form $y^2=q(x)$, where $q$ is a quintic polynomial in $x$. Then the space $\Omega^1_C$ of holomorphic differentials on $C$ is the complex vector space spanned by $\frac{dx}{y}$ and $x\frac{dx}{y}$. Hence, any holomorphic differential on $C$ has the form $(ax+b)\frac{dx}{y}$. Now, $\frac{dx}{y}$ has a double zero at $\infty$ and no other zeros or poles. If $-\frac{b}{a}$ is not a root of $q(x)$ (equivalently, if $(-\frac{b}{a},q(-\frac{b}{a})^{1/2})$ is not a Weierstra\ss\ point of $C$), then $ax+b$ has simple zeros at the two points on $C$ with $x$-coordinate $-\frac{b}{a}$, and a double pole at $\infty$. Hence, $(ax+b)\frac{dx}{y}$ has simple zeros at two points and hence two critical points. Thus, the only holomorphic differentials with a double zero at some point of $C$ are those with the double zero at a Weierstra\ss\ point. This yields the desired result. \end{proof}

Written in the form $y^2=q(x)$, where $q$ is a quintic polynomial in $x$, $C$ has a Weierstra\ss\ point at $\infty$, So, let us also decide that, if $\omega=\frac{dx}{y}$ is the invariant differential on $E$, then $f^\ast\omega=a\frac{dx}{y}\in\Omega^1_C$, for some $a\in\CC^\times$. This is equivalent to asserting that the triple point of $f$ above the point at $\infty$ of $E$ is the point at $\infty$ at $(0:1:0)\in C$. Hence, such a map will be a cover of the desired form.

Now, suppose we have such a map $f$. Let $\alpha$ be a loop in $C$ (or an element of $H_1(C,\ZZ)$, if we prefer). Then we have \[\int_\alpha f^\ast\omega=\int_{f_\ast\alpha} \omega.\] These integrals are known as periods. Now, since $H_1(E,\ZZ)$ is $\ZZ^2$, the periods of $E$ form a lattice in $\CC$. But since $H_1(C,\ZZ)$ is 4-dimensional, the periods of $\frac{dx}{y}\in\Omega^1_C$ are typically dense in $\CC$. In order to find a curve $C$ that admits such a map, it is necessary that its periods form a lattice, and in fact a lattice homothetic to that formed by $E$.

From Figure \ref{swisscross}, we can determine the induced map on $H_1$. We first choose preferred homology bases on $C$ and $E$, and then we determine how they appear on the Swiss cross diagram. \begin{figure} \begin{center} \includegraphics[height=1.5in]{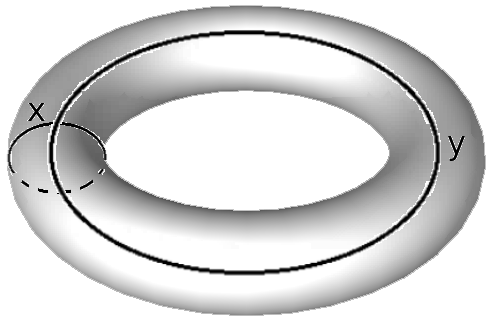} \end{center} \caption{An elliptic curve with a preferred basis of $H_1$ shown.} \label{torus} \end{figure} \begin{figure} \begin{center} \includegraphics[height=1.5in]{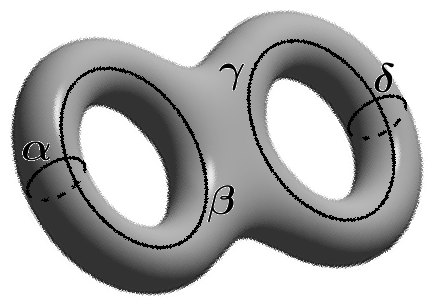} \end{center} \caption{A genus 2 curve with a typical basis of $H_1$ shown.} \label{genus2} \end{figure} Figure \ref{torus} pictures a basis of $H_1$ on $E$, and Figure \ref{genus2} pictures a basis of $H_1$ on $C$. We wish to arrange for $f$ to behave as follows with respect to these bases: \begin{align*} f_\ast(\alpha) &= m_1x \\ f_\ast(\beta) &= n_1y \\ f_\ast(\gamma) &= n_2y \\ f_\ast(\delta) &= m_2x, \end{align*} for certain integers $m_1,m_2,n_1,n_2$. To determine these integers, we draw the bases of $H_1(C)$ and $H_1(E)$ on our square diagrams. From Figures \ref{torus} and \ref{genus2}, we know the intersection pairing of the homology bases, and it suffices to find a basis on the square diagrams with the same intersection pairings and mapping properties. The correct pictures on the square diagrams are shown in Figures \ref{EChomology} and \ref{genus2homology}. \begin{figure} \begin{center} \includegraphics[height=.6in]{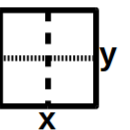} \end{center} \caption{A preferred homology basis of $H_1(E)$.} \label{EChomology} \end{figure} \begin{figure} \begin{center} \includegraphics[height=1.5in]{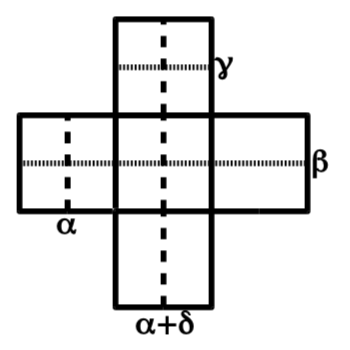} \end{center} \caption{A preferred homology basis of $H_1(C)$, drawn on the Swiss cross diagram.} \label{genus2homology} \end{figure} From these diagrams, we can see that \begin{align*} f_\ast\alpha &= x, \\ f_\ast\beta &= 3y, \\ f_\ast\gamma &= y, \\ f_\ast(\alpha+\delta)&=3x,\end{align*} and the first and fourth equations imply that \[f_\ast\delta=2x.\]

Now, we specialize to the elliptic curve $E:y^2=x^3-x$, which has $j$-invariant 1728. Computations with this elliptic curve are somewhat nicer than they are with arbitrary elliptic curves because its period lattice is a square lattice. Hence, we can take its fundamental periods to be $\varpi$ and $\varpi i$, where $\varpi\approx 5.2441151$.

We now identify $\alpha,\beta,\gamma,\delta$ with representatives of their respective homology classes which pass through two of the Weierstra\ss\ points of $C$, i.e., the preimages of $r_1$, $r_2$, $r_3$, $r_4$, $r_5$, and $\infty$ under the natural hyperelliptic map $C\to\PP^1$. We can therefore draw the images of $\alpha,\beta,\gamma,\delta$ under the hyperelliptic map in $\PP^1$; this picture is given in Figure \ref{genus2branch}. In order to avoid sign problems, we choose to orient the unlabeled gray loop in Figure \ref{genus2branch} clockwise and the others counterclockwise. When we do this, then the sum of the three homology classes corresponding to the gray loops is 0.

Figure \ref{genus2branch} shows us how the loops in Figure \ref{genus2} are related to periods of $\frac{dx}{y}\in\Omega^1_C$. The integral $\int_{r_1}^{r_2}\frac{dx}{y}$ is $\frac{1}{2}\int_{\alpha}\frac{dx}{y}$, for example. Since the sum of the three gray loops is $0\in H_1(C,\ZZ)$, we have \[\int_{r_1}^{r_2}\frac{dx}{y}-\int_{r_3}^{r_4}\frac{dx}{y}+\int_{r_5}^\infty\frac{dx}{y}=0,\] or \[\int_{r_3}^{r_4}\frac{dx}{y}=\int_{r_1}^{r_2}\frac{dx}{y}+\int_{r_5}^\infty\frac{dx}{y}=3\int_{r_1}^{r_2}\frac{dx}{y}\] since \[\int_{r_5}^\infty\frac{dx}{y}=2\int_{r_1}^{r_2}\frac{dx}{y}.\]

\begin{figure} \begin{center} \includegraphics[height=1.5in]{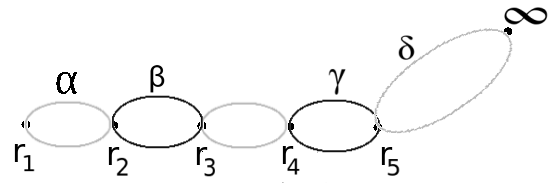} \end{center} \caption{The data of Figure \ref{genus2} has been rearranged. For example, the integral of a form over $\alpha$ is twice the integral from $r_1$ to $r_2$. This figure and Figure \ref{genus2} have the same symplectic intersection matrix. The unlabeled gray loop is oriented clockwise, whereas the others are oriented counterclockwise.} \label{genus2branch} \end{figure}

Now, we switch to a slightly different homology basis, so that we can consider the integrals \[\int_{r_j}^{r_{j+1}} \frac{dx}{y}\] for $j=1,2,3,4$. In particular, we are getting rid of $\delta$ in favor of the unlabeled gray loop in Figure \ref{genus2branch}. So, we wish to find a genus-2 curve for which these periods with respect to the differential $\frac{dx}{y}$ are proportional to $\varpi$, $3\varpi i$, $3\varpi$, and $\varpi i$. That is, if we let $C$ be the curve \[C:y^2=(x-r_1)(x-r_2)(x-r_3)(x-r_4)(x-r_5),\] then we have \begin{equation} I_j := 2\int_{r_j}^{r_{j+1}} \frac{dx}{\sqrt{(x-r_1)(x-r_2)(x-r_3)(x-r_4)(x-r_5)}}=Km_j\label{intrelations}\varpi\end{equation} for some constant $K$ and for $j=1,2,3,4$, where $m_1=1$, $m_2=3i$, $m_3=3$, and $m_4=i$. Notice now that if \begin{equation} r_1+r_5=r_2+r_4=2r_3, \label{rchoice} \end{equation} then we have \begin{align*} I_4 &= iI_1, \\ I_2 &= iI_3.\end{align*} so we need only arrange to have $I_3=3I_1$ for (\ref{intrelations}) to hold.

Now, it is not \emph{a priori} clear that we can find a curve $C$ whose underlying quintic has five real roots, or one for which (\ref{rchoice}) holds. Still, it is reasonable to begin the search by looking for such a $C$.

After this simplification, we see that $C$ is expected to have the form \[C:y^2=x(x-1)(x-\kappa)(x-2\kappa+1)(x-2\kappa)\] for some $\kappa$. The task of finding $\kappa$ remains. By the results of \S\ref{existence}, we know that $\kappa\in\Qbar$, if it exists at all.

We now find $\kappa$ to high precision using a Newton's method algorithm. We set \begin{align*} \pi_1(s) &= \int_0^1 \frac{dx}{\sqrt{x(x-1)(x-s)(x-2s+1)(x-2s)}}, \\ \pi_3(s) &= \int_s^{2s-1} \frac{dx}{\sqrt{x(x-1)(x-s)(x-2s+1)(x-2s)}}, \end{align*} and set $\rho(s) = \pi_3(s)/\pi_1(s)$. We search for an $s$ with $\rho(3)=3$.

To do this, we start with a seed value $s_0$ and define \begin{equation} s_{n+1} = s_n-\frac{\rho(s_n)-3}{\rho'(s_n)}. \label{newton} \end{equation} Then, assuming we have chosen an appropriate $s_0$, then $s_n$ converges to some $s$ with $\rho(s)=3$. Fortunately, the basin of attraction for this root contains the interval $(1,400)$, so it is very easy to choose a seed that converges.

In order to perform this computation, we need to be able to compute integrals to high precision. Fortunately, this can be done quite rapidly by using tanh-sinh quadrature, as described in \cite{BB11}. Furthermore, rather than attempting to compute $\rho'(s_n)$ in (\ref{newton}) symbolically, we pick some very small $\eps$ and use \[\rho'(s_n) \approx \frac{\rho(s_n+\eps)-\rho(s_n)}{\eps}.\]

If we start from a seed value of $s_0=2$ (which is a very poor guess, but one which admits a convergent sequence of $s_n$'s nonetheless) and take $\eps=10^{-10}$, we find that \begin{align*} s_1 &\approx 8.7404, \\ s_2 &\approx 33.2055, \\ s_3 &\approx 85.3301, \\ s_4 &\approx 139.6842, \\ s_5 &\approx 159.9489, \\ s_6 &\approx 161.4910, \\ s_7 &\approx 161.4984, \end{align*} and so forth. This sequence converges to \[\kappa\approx 161.49844718999242907073.\] The task remaining is to determine which algebraic number $\kappa$ is, i.e.\ to find a nonzero polynomial in $\ZZ[x]$ having $\kappa$ as a root.

%$Using a Newton's method algorithm, we can compute $\kappa$ to high precision, assuming that we can compute periods accurately. Fortunately, this can be done quite rapidly by using tanh-sinh quadrature, as described in \cite{BB11}. Using this method, we can compute $\kappa$ to hundreds of digits. This computation yields a number $\kappa$, whose first few digits are \[\kappa\approx 161.49844718999242907073,\] The task remaining is to determine which algebraic number $\kappa$ is, i.e.\ to find a nonzero polynomial in $\ZZ[x]$ having $\kappa$ as a root.

\section{Determining $\kappa$ explicitly} \label{algebraize}

Let us, at this stage, make a brief digression to discuss how we can recognize certain algebraic numbers given their decimal expansions to several digits. A first question to ask is, given some number $r\in\RR$, how can we tell whether $r\in\QQ$? One possible method is to look for periodicity in the decimal expansion of $r$, but it might be that the period is quite large, and we lack sufficient precision to do that. A better way is to look at the continued fraction expansion of $r$; a number is rational if and only if its continued fraction expansion terminates. Given a truncation of the decimal expansion of $r$, we can predict that $r$ may be rational if one of its terms is very large.

\begin{example} Consider a number $r$ whose decimal expansion begins $1.88372093$. The continued fraction expansion based on the data we have is \[[1;1,7,1,1,2,2325581].\] Hence, we can guess that $r$ might be the number whose continued fraction expansion is \[[1;1,7,1,1,2],\] or $81/43$. Working out more digits can help to confirm this guess. \end{example}

Similarly, a real number $r$ is a root of a quadratic polynomial in $\ZZ[x]$ if and only if its continued fraction expansion is eventually periodic. This again we can recognize based on a truncation of its decimal expansion, although we need more digits in this case.

\begin{example} Let $r$ be a real number whose decimal expansion begins $64.89974874212324$. Its continued fraction then is roughly \[[64;1,8,1,38,1,8,1,38,1,8,1,42,\cdots].\] We may then guess that the 42 should really be a 38, and that the continued fraction expansion of $r$ might be $[64;\overline{1,8,1,38}]$, so that $r=45+6\sqrt{11}$. Again, computing more digits can help to confirm this guess. We require somewhat more digits to detect periodicity of the continued fraction expansion than we do for detecting rationals by terminating continued fraction expansion, as the length of the period can be quite large; for example, the period of $\sqrt{43}$ is already of length 10, and we would be likely not to detect it with under 13 digits of accuracy. Indeed, the {\tt algdep} function in PARI/GP \cite{PARI2} requires 17 digits in order to detect this number. \end{example}

For higher-degree irrational numbers, we can again recognize them (using a method described in \cite{KZ01}), although it takes far more digits to do so. To see if $r$ is a root of a degree-$n$ polynomial, we can use the LLL algorithm to try to find a vector $(a_0,\ldots,a_n)\in\ZZ^{n+1}\setminus\{0\}$ so that \[(a_nr^n+a_{n-1}r^{n-1}+\cdots+a_0)^2+\eps(a_n^2+\cdots+a_0^2)\] is very small, for a suitably small $\eps>0$. If we can do so, then $r$ may be a root of the polynomial $a_nx^n+a_{n-1}x^{n-1}+\cdots+a_0$.

Let us estimate the number of digits we ought to use in order for this method to apply, following \cite{BL00}. Suppose that we have a real number $r$, which is a root of a polynomial $f$ of degree at most $n$, with coefficients of at most $d$ digits. Then there are $(2\cdot 10^d+1)^{n+1}$ possible polynomials, so we expect to be able to match roughly $(d+1)(n+1)$ digits. Thus we need to have more than $(d+1)(n+1)$ digits in order to expect that we have found the correct polynomial.

Now let us return to the task of determining $\kappa$. The continued fraction of $\kappa$ is roughly \[[161;2,160,2,160,2,160,2,160,2,7,\ldots].\] A good guess, then, is that $\kappa=[161;\overline{2,160}]=81+36\sqrt{5}$. Indeed, computation of more digits helps to increase confidence in this guess.

It is worth mentioning here that this method does not \emph{prove} that $\kappa=81+36\sqrt{5}$ works, since this number might merely be exceptionally close to the true value of $\kappa$. However, once we have the equation for the map as well, we will be able to check that this number is actually correct if we desire to do so. In the opinion of the author, this should not be seen as a defect of the method any more than having to verify that a clever approach to solving any other problem works by justifying our approach \emph{a posteriori} by presenting a complete solution as we eventually do in the above proof.

\begin{remark} We might hope to be able to confirm that our guess of $\kappa$ is correct without explicitly producing the equation of the map and providing a proof in the style of that of Theorem \ref{311curve} in \S\ref{thmstatement}. Implicit in the results of \S\ref{existence} is a bound $D$ on the degree of $\kappa$. Since for any $B>0$ there are only finitely many algebraic numbers of degree at most $D$ with height at most $B$, it would suffice to produce a bound on the possible height of $\kappa$, so that computing $\kappa$ to sufficiently many digits could certify that it is what we believe it to be as an algebraic number. At the moment, we do not know how to produce such a height bound, but such an exploration would be interesting. \end{remark}

\section{The equation of the map} \label{eqnmap}

Now, working on the assumption that $\kappa=81+36\sqrt{5}$, we compute the defining equation of the map $f:C\to E$. Let $\Lambda_E$ and $\Lambda_C$ be the period lattices for $E$ and $C$, respectively. Then $\Lambda_E$ and $\Lambda_C$ are homothetic lattices in $\CC$. The map $C\to\CC/\Lambda_C$ is given by \begin{align*} P &\mapsto\int_{(-\infty,0)}^P \frac{dx}{y}\mod\Lambda_C \\ &=\pm\int_{-\infty}^{x(P)} \frac{dx}{\sqrt{x(x-1)(x-\kappa)(x-2\kappa+1)(x-2\kappa)}}\mod\Lambda_C,\end{align*} where $x(P)$ is the $x$-coordinate of $P$ and the sign depends on which preimage of $x(P)$ under the hyperelliptic map $C\to\PP^1$ is chosen. The map $\CC/\Lambda_E\to E$ is given by $z\mapsto(\wp(z),\wp'(z))$, where $\wp$ and $\wp'$ are the Weierstra\ss\ $\wp$-function for the lattice $\Lambda_E$ and its derivative. There is also a scaling map $\CC/\Lambda_C\to\CC/\Lambda_E$. Composing these maps gives us the map $C\to E$. If we know the form of the map $f$ and values of $f(P)$ at enough points, then we can solve a system of linear equations to determine the coefficients to high precision. Then, just as before, we can convert the decimal expansions into concrete algebraic numbers.

%Now, working on the assumption that $\kappa=81+36\sqrt{5}$, we compute the defining equation of the map $f:C\to E$. Let $\Lambda_E$ be the period lattice for $E$ and $\Lambda_C$ the period lattice for $C$, so that $\Lambda_E$ and $\Lambda_C$ are homothetic lattices in $\CC$. Analytically, $E\cong\CC/\Lambda_C$, and the map $f:C\to\CC/\Lambda'$ is given by \[f(P)=\int_{(-\infty,0)}^P \frac{dx}{\sqrt{x(x-1)(x-\kappa)(x-2\kappa+1)(x-2\kappa)}}.\] We can then convert this to a point on $E$ by using the elliptic exponential function, which is given in terms of the Weierstra\ss\ $\wp$-function and its derivative for the lattice $\Lambda'$: the map $\CC/\Lambda\to E$ is given by $z\mapsto(\wp(z),\wp'(z))$. If we know the form of the map $f$ and values of $f(P)$ at enough points, then we can compose with the map $\CC/\Lambda\to E$ and solve a system of linear equations to determine the map, with coefficients being decimal expansions. Then, just as before, we can convert the decimal expansions into concrete algebraic numbers.

Before we compute the map from $C$ to $E$, we change variables slightly, letting our genus-2 curve be \[-\sqrt{5}y^2=x(x-1)(x-\kappa)(x-2\kappa+1)(x-2\kappa).\] This does not change the isomorphism class of the curve, nor does it change the induced map $\PP^1\to\PP^1$ induced by the hyperelliptic involutions on $C$ and $E$, but it does make the coefficients of the $y$-coordinate (slightly) cleaner, casuing them to lie in $\QQ(\sqrt{5})$ rather than a quadratic extension of $\QQ(\sqrt{5})$.

The map $f:C\to E$ can be written as $f(x,y)=(g(x,y),h(x,y))$, where $g,h:C\to\PP^1$. Furthermore, the map $x:E\to\PP^1$ which remembers the $x$-coordinate has degree 2, and the map $y:E\to\PP^1$ which remembers the $y$-coordinate has degree 3. Hence, $g$ has degree 10, and $h$ has degree 15. Also, from the square diagram of Figure \ref{swisscross}, we can see that $f$ respects the hyperelliptic involution: the elliptic involution on the elliptic curve corresponds to a rotation by $\pi$ of the fundamental parallelogram; similarly, the hyperelliptic involution on the genus-2 curve corresponds to a rotation by $\pi$ of the Figure \ref{swisscross}. Hence, $f(x,-y)=(g(x,y),-h(x,y))$, so $g(x,y)$ only depends on $x$, and $h(x,y)/y$ only depends on $x$. Finally, $g$ has two simple poles. This allows us to determine that we can write \[g(x,y)=\frac{x^5+a_4x^4+a_3x^3+a_2x^2+a_1x+a_0}{b_2x^2+b_1x+b_0}.\] By approximating $g$ at enough points on $C$ to sufficiently high precision, we can solve for the coefficients $a_i$ and $b_i$ and algebraize as we did in \S\ref{algebraize}, and we find that \begin{align*} a_4 &= -45(9+4\sqrt{5}) \\ a_3 &= 660(161+72\sqrt{5}) \\ a_2 &= -3240(2889+1292\sqrt{5}) \\ a_1 &= 1980(51841+23184\sqrt{5}) \\ a_0 &= -324(930249+416020\sqrt{5}) \\ b_2 &= -100\sqrt{5}(2889+1292\sqrt{5}) \\ b_1 &= 1800\sqrt{5}(51841+23184\sqrt{5}) \\ b_0 &= 0.\end{align*}

\begin{remark} The coefficients in the map are surprisingly nice. Let $\eps=\frac{1+\sqrt{5}}{2}$; this is a fundamental unit in $\QQ(\sqrt{5})$. Then the coefficients in the map are (essentially) integral multiples of powers of $\eps$. We have \begin{align*} \eps^6 &= 9+4\sqrt{5}, \\ \eps^{12} &= 161+72\sqrt{5} \\ \eps^{18} &= 2889+1292\sqrt{5} \\ \eps^{24} &= 51841+23184\sqrt{5} \\ \eps^{30} &= 930249+416020\sqrt{5} \\ \eps^{36} &= 16692641+7465176\sqrt{5}. \end{align*} Hence, there is some homogeneity in the map: if we replace $x$ by $\eps^6x$, then all the coefficients are either integers or integral multiples of $\sqrt{5}$. This seems unlikely to be a coincidence, but at the moment, we lack an explanation for this phenomenon. \end{remark}

To compute $h(x,y)$, we recall that $g$ and $h$ satisfy the relation $-\sqrt{5}h^2=g^3-g$. Hence, we compute $g^3-g$. Also, we know that $h(x,y)$ can be written as $h_1(x)y$, so \[h_1(x)^2=\frac{g(x)^3-g(x)}{x(x-1)(x-\kappa)(x-2\kappa+1)(x-2\kappa)}.\] We therefore have \[h(x,y)=\frac{(x^6+c_5x^5+c_4x^4+c_3x^3+c_2x^2+c_1x+c_0)y}{dx^2 (x-2\kappa)^2},\] where \begin{align*} c_5 &= -54(9+4\sqrt{5}) \\ c_4 &= 1030(161+72\sqrt{5}) \\ c_3 &= -7920(2889+1292\sqrt{5}) \\ c_2 &= 18780(51841+23184\sqrt{5}) \\ c_1 &= 216(930249+416020\sqrt{5}) \\ c_0 &= -1944(16692641+7465176\sqrt{5}) \\ d &= 1000\sqrt{5}(219602+98209\sqrt{5}) \end{align*}

With everything now so concrete, we can verify that all our numerics were correct: that $f$ does in fact define a map $C\to E$, branched only above $\infty$, and with a triple point in $C$. This completes the construction of a curve $C$ which admits a map to $E$.

\section{Further Results and Possibilities} \label{further}

\subsection{Covers with other branching types}

The same method can be used to compute some branched covers of elliptic curves with other branching types and degrees. For example, there is a genus-2 curve $C$ admitting a map of degree 7 to the elliptic curve $y^2=x^3-x$, branched above $\infty$, with one triple point above $\infty$. The computations involved in finding the equation of this curve $C$ are similar to those given above, and an equation for $C$ has the same form \[C:y^2=x(x-1)(x-\kappa')(x-2\kappa'+1)(x-2\kappa'),\] where now $\kappa'=932+352\sqrt{7}+18\sqrt{5355+2024\sqrt{7}}$.

For certain branching types, there are different methods that can be used for computing branched covers in practice. The author in \cite{RS12} has shown how to produce covers for all elliptic curves and all odd degrees in the totally ramified case.

\subsection{Covers of other elliptic curves}

We can ask about computing covers of other elliptic curves as well. For any particular elliptic curve, computing such a cover should be similar to this case, some of the details may be somewhat more difficult. In particular, we will not be able to assume that $C$ has the symmetric form $y^2=x(x-1)(x-\kappa)(x-2\kappa+1)(x-2\kappa)$, as we did here; instead, we will need to use Newton's method on a system of three equations at once. However, this does not present any real difficulties. For other elliptic curves, though, we should not expect the coefficients to be defined over a quadratic extension of $\QQ$; the degree may be higher. Hence, it will require considerably more precision in order to determine their minimal polynomials.

Another approach is to try to find a family of genus-2 curves (in fact, a curve inside the moduli space $\mcM_2$) covering a family of elliptic curves. One way of doing this is by starting with the cover we found here and deforming it into a family. The author has made some steps in this direction in his PhD thesis \cite{RSthesis}, but the computations eventually take too long. We hope to improve the process so as to complete this task.

\section{Computations}

Computations for this paper were done using Mathematica \cite{Mathematica}, mpmath \cite{mpmath}, PARI/GP \cite{PARI2}, Sage \cite{sage}, and SymPy \cite{Sympy}.

\section*{Acknowledgements}

The author would like to thank Akshay Venkatesh for suggesting this problem and for many helpful discussions. He would also like to thank Jeffrey Danciger, Ilya Grigoriev, Michael Lipnowski, Jeremy Miller, and Ronen Mukamel for providing useful suggestions. He would also like to thank the referee for many helpful comments and for suggesting this proof of Lemma \ref{Wptlemma}.

\bibliographystyle{alpha}
\bibliography{origami311-2}

\newcommand{\etalchar}[1]{$^{#1}$}
\begin{thebibliography}{{Sym}12}

\bibitem[Bak67]{Baker67}
A.~Baker.
\newblock Linear forms in the logarithms of algebraic numbers. {I}, {II},
  {III}.
\newblock {\em Mathematika 13 (1966), 204-216; ibid. 14 (1967), 102-107;
  ibid.}, 14:220--228, 1967.

\bibitem[BB11]{BB11}
D.~H. Bailey and J.~M. Borwein.
\newblock High-precision numerical integration: progress and challenges.
\newblock {\em J. Symbolic Comput.}, 46(7):741--754, 2011.

\bibitem[Bec89]{Beck89}
Sybilla Beckmann.
\newblock Ramified primes in the field of moduli of branched coverings of
  curves.
\newblock {\em J. Algebra}, 125(1):236--255, 1989.

\bibitem[BL00]{BL00}
Jonathan~M. Borwein and Petr Lison{\v{e}}k.
\newblock Applications of integer relation algorithms.
\newblock {\em Discrete Math.}, 217(1-3):65--82, 2000.
\newblock Formal power series and algebraic combinatorics (Vienna, 1997).

\bibitem[BvdP95]{BvdP95}
Enrico Bombieri and Alfred~J. van~der Poorten.
\newblock Continued fractions of algebraic numbers.
\newblock In {\em Computational algebra and number theory ({S}ydney, 1992)},
  volume 325 of {\em Math. Appl.}, pages 137--152. Kluwer Acad. Publ.,
  Dordrecht, 1995.

\bibitem[CG94]{CG94}
Jean-Marc Couveignes and Louis Granboulan.
\newblock Dessins from a geometric point of view.
\newblock In {\em The {G}rothendieck theory of dessins d'enfants ({L}uminy,
  1993)}, volume 200 of {\em London Math. Soc. Lecture Note Ser.}, pages
  79--113. Cambridge Univ. Press, Cambridge, 1994.

\bibitem[DD97]{DD97}
Pierre D{\`e}bes and Jean-Claude Douai.
\newblock Algebraic covers: field of moduli versus field of definition.
\newblock {\em Ann. Sci. \'Ecole Norm. Sup. (4)}, 30(3):303--338, 1997.

\bibitem[Don11]{Don11}
Simon Donaldson.
\newblock {\em Riemann surfaces}, volume~22 of {\em Oxford Graduate Texts in
  Mathematics}.
\newblock Oxford University Press, Oxford, 2011.

\bibitem[HS09]{HS09}
Frank Herrlich and Gabriela Schmith{\"u}sen.
\newblock Dessins d'enfants and origami curves.
\newblock In {\em Handbook of {T}eichm\"uller theory. {V}ol. {II}}, volume~13
  of {\em IRMA Lect. Math. Theor. Phys.}, pages 767--809. Eur. Math. Soc.,
  Z\"urich, 2009.

\bibitem[J{\etalchar{+}}10]{mpmath}
Fredrik Johansson et~al.
\newblock {\em mpmath: a {P}ython library for arbitrary-precision
  floating-point arithmetic (version 0.14)}, February 2010.
\newblock {\tt http://code.google.com/p/mpmath/}.

\bibitem[KZ01]{KZ01}
Maxim Kontsevich and Don Zagier.
\newblock Periods.
\newblock In {\em Mathematics unlimited---2001 and beyond}, pages 771--808.
  Springer, Berlin, 2001.

\bibitem[LLL82]{LLL82}
A.~K. Lenstra, H.~W. Lenstra, Jr., and L.~Lov{\'a}sz.
\newblock Factoring polynomials with rational coefficients.
\newblock {\em Math. Ann.}, 261(4):515--534, 1982.

\bibitem[McM05]{McM05}
Curtis~T. McMullen.
\newblock Teichm\"uller curves in genus two: the decagon and beyond.
\newblock {\em J. Reine Angew. Math.}, 582:173--199, 2005.

\bibitem[PAR10]{PARI2}
{\em {PARI/GP, version {\tt 2.3.5}}}.
\newblock Bordeaux, 2010.
\newblock \url{http://pari.math.u-bordeaux.fr/}.

\bibitem[Rob04]{Rob04}
David~P. Roberts.
\newblock An {$ABC$} construction of number fields.
\newblock In {\em Number theory}, volume~36 of {\em CRM Proc. Lecture Notes},
  pages 237--267. Amer. Math. Soc., Providence, RI, 2004.

\bibitem[RS12a]{RSthesis}
Simon Rubinstein-{S}alzedo.
\newblock Controlling ramification in number fields.
\newblock 2012.
\newblock PhD thesis. \url{http://www.math.dartmouth.edu/~simon/thesis.pdf}.

\bibitem[RS12b]{RS12}
Simon Rubinstein-{S}alzedo.
\newblock Totally ramified branched covers of elliptic curves.
\newblock 2012.
\newblock \url{http://arxiv.org/abs/1210.3195}.

\bibitem[S{\etalchar{+}}10]{sage}
W.\thinspace{}A. Stein et~al.
\newblock {\em {S}age {M}athematics {S}oftware ({V}ersion 4.5.1)}.
\newblock The Sage Development Team, 2010.
\newblock {\tt http://www.sagemath.org}.

\bibitem[{Sym}12]{Sympy}
{SymPy Development Team}.
\newblock {\em SymPy: Python library for symbolic mathematics}, 2012.

\bibitem[{Wol}07]{Mathematica}
{Wolfram Research, Inc.}, Champaign, Illinois.
\newblock {\em {Mathematica, Version 6.0.1.0}}, 2007.

\bibitem[Zor06]{Zor06}
Anton Zorich.
\newblock Flat surfaces.
\newblock In {\em Frontiers in number theory, physics, and geometry. {I}},
  pages 437--583. Springer, Berlin, 2006.

\end{thebibliography}

\end{document}